\documentclass{amsart}

\input xy
\xyoption{all}

\usepackage{geometry}
\usepackage{amsmath}
\usepackage{amssymb}
\usepackage{cite}
\usepackage{amsthm}  % see geometry.pdf on how to lay out the page. There's lots.
\usepackage{tikz}
\usetikzlibrary{er,positioning}

\newtheorem{same}{This should never appear}[section]
\newtheorem{defin}[same]{Definition}

\newtheorem{remark}[same]{Remark}
\newtheorem{theorem}[same]{Theorem}
\newtheorem{example}[same]{Example}
\newtheorem{lemma}[same]{Lemma}
\newtheorem{fact}[same]{Fact}
\newtheorem{question}[same]{Question}
\newtheorem{cor}[same]{Corollary}

\newtheorem{nota}[same]{Notation}
\newtheorem*{prob}{Main Problem}

\newtheorem{defin*}{Definition}
\newtheorem*{theorem*}{Theorem}

\makeatletter
\newcommand{\skipitems}[1]{%
  \addtocounter{\@enumctr}{#1}%
}
\newcommand{\bb}{\mathbf{b}}
\newcommand{\rest}{\mathord{\upharpoonright}}
\newcommand{\id}{\textrm{id}}

\newcommand{\K}{\mathbf{K}}

\newcommand{\Kp}{\K^{p\text{-grp}}}
\newcommand{\Kkp}{K^{p\text{-grp}}}

\newcommand{\Kab}{\K^{Ab}}

\newcommand{\LS}{\operatorname{LS}}

\newcommand{\leap}[1]{\le_{#1}}

\newcommand{\lea}{\leap{\K}}

\newcommand{\gtp}{\mathbf{gtp}}
\newcommand{\gS}{\mathbf{gS}}

\DeclareMathOperator{\pp}{pp}    % pp type
\title{A model theoretic solution to a problem of L\'{a}szl\'{o} Fuchs}
\date{\today.} % delete this line to display the current date

\author{Marcos Mazari-Armida}
\email{mmazaria@andrew.cmu.edu}
\urladdr{http://www.math.cmu.edu/~mmazaria/ }
\address{Department of Mathematical Sciences \\ Carnegie Mellon
University \\ Pittsburgh, Pennsylvania, USA}

\begin{document}

%%%%%%%%%%%%%%%%%%%%%%%%%%%%%%%%%%%%%%%%%%%%%%%
\begin{abstract}

Problem 5.1 in page 181 of \cite{fuc} asks to find the cardinals $\lambda$ such that there is a universal abelian $p$-group for purity of cardinality $\lambda$, i.e., an abelian $p$-group $U_\lambda$ of cardinality $\lambda$ such that every abelian $p$-group of cardinality $\leq \lambda$ purely embeds in $U_\lambda$.  In this paper we use ideas from the theory of abstract elementary classes to show:
\begin{theorem}
Let $p$ be a prime number.
If $\lambda^{\aleph_0}=\lambda$ or $\forall \mu < \lambda(
\mu^{\aleph_0} < \lambda)$, then there is a universal abelian $p$-group for purity of cardinality $\lambda$. Moreover for $n\geq 2$, there is a universal abelian $p$-group for purity of cardinality $\aleph_n$ if and only if $2^{\aleph_0} \leq \aleph_n$. 
\end{theorem}

As the theory of abstract elementary classes has barely been used to tackle algebraic questions,  an effort was made to introduce this theory from an algebraic perspective.

\end{abstract}

%%%%%%%%%%%%%%%%%%%%%%%%%%%%%%%%%%%%%%%%%%%%%%%

\maketitle

{\let\thefootnote\relax\footnote{{AMS 2010 Subject Classification:
Primary: 20k30, 03C48. Secondary: 03C45, 03C60, 13L05.
Key words and phrases.  Abelian groups; p-groups; Abstract Elementary
Classes; Universal models.}}}

%\tableofcontents

%%%%%%%%%%%%%%%%%%%%%%%%%%%%%%%%%%%%%%%%%%%%%%%
\section{Introduction}

The aim of this paper is to address Problem 5.1 in page 181 of \cite{fuc}. The problem stated by Fuchs is the following:

\begin{prob}\label{mprob}
For which cardinals $\lambda$ is there a universal abelian $p$-group for purity? We mean an abelian $p$-group $U_\lambda$ of cardinality $\lambda$ such that every abelian $p$-group of cardinality $\leq \lambda$ embeds in $U_\lambda$ as a pure subgroup. The same question for torsion-free abelian groups.
\end{prob}

The question for torsion-free abelian groups has been thoroughly studied as witnessed by \cite{kojsh}, \cite{sh3}, \cite{sh820} and \cite{kuma}. Due to this, we focus in this paper in the case of abelian $p$-groups\footnote{Recall that $G$ is an abelian $p$-group if it is an abelian group and every element of $G$ different from zero has order $p^n$ for some $n \in \mathbb{N}$.}. Regarding abelian $p$-groups, there are some results for the  subclass of separable abelian $p$-groups, these results appeared in \cite{kojsh}, \cite{sh3} and \cite{sh820}.

The solution we provide for the Main Problem extends the ideas presented in \cite{kuma} to the class of abelian $p$-groups.\footnote{ For the reader familiar with first-order logic and \cite{kuma}, the reason we can not simply apply the results of \cite{kuma} to abelian $p$-groups is because the class of abelian $p$-groups is not first-order axiomatizable.} We show that there are many cardinals with universal abelian $p$-groups for purity:

\textbf{ Theorem \ref{main}. } \textit{Let $p$ be a prime number.
If $\lambda^{\aleph_0}=\lambda$ or $\forall \mu < \lambda(
\mu^{\aleph_0} < \lambda)$, then there is a universal abelian $p$-group for purity of cardinality $\lambda$.}

The proof has three main steps.\footnote{All the definitions needed to understand this paragraph are given in Section 2 with many algebraic examples to showcase them.} First, we identify the class of abelian $p$-groups with pure embeddings as an abstract elementary class and show that it has amalgamation, joint embedding and no maximal models (see Fact \ref{bas}). Then, we show that the class of abelian $p$-groups with pure embeddings  is $\lambda$-stable if $\lambda^{\aleph_0}=\lambda$ (see Theorem \ref{psta}). Finally, the assertion follows from using some general results on abstract elementary classes (see Theorem  \ref{main}). 

Using some results of \cite{kojsh} it is possible to show that if some cardinal inequalities hold, then there are some cardinals where there can not be universal abelian $p$-groups for purity (see Lemma \ref{nouniv}). The techniques used to obtain this result are explained in detail in \cite{kojsh}, \cite{dz} and \cite[\S 2.4]{balc}.

As a simple corollary of Theorem \ref{main} and what was mentioned in the paragraph above, we obtain a complete solution to the Main Problem below $\aleph_\omega$ with the exception of $\aleph_0$ and $\aleph_1$.

\textbf{Corollary \ref{bome}.} \textit{Let $p$ be a prime number.
For $n\geq 2$, there is a universal abelian $p$-group for purity of cardinality $\aleph_n$ if and only if $2^{\aleph_0} \leq \aleph_n$.}

We address the case of $\aleph_1$ in Lemma \ref{one} and show that the answer depends on the value of the continuum and a combinatorial principle. We leave open the case of $\aleph_0$ (see Question \ref{quest} and the remark below it). 

In Section 2, we make an effort to present all the necessary notions of abstract elementary classes that are needed to understand the proof of the main theorem (Theorem \ref{main}). We present them from an algebraic perspective and give many examples. In particular, we do not assume that the reader is familiar with logic. 

The paper is divided into four sections. Section 2 presents an introduction to abstract elementary classes from an algebraic point of view. Section 3  has the main results. Section 4 presents how the main results can be generalized to other classes.

This paper was written while the author was working on a Ph.D. under the direction of Rami Grossberg at Carnegie Mellon University and I would like to thank Professor Grossberg for his guidance and assistance in my research in general and in this work in particular. I would like to thank John T. Baldwin, Thomas Kucera and Samson Leung for comments that helped improve the paper. I am grateful to the referee for many comments that significantly improved the presentation of the paper. 

\section{An introduction to AECs from an algebraic point of view}

This section will introduce the basic notions of abstract elementary classes that are used in this paper and will hopefully motivate the use of abstract elementary classes to tackle algebraic questions. This section assumes no familiarity with logic, with the exception of Example \ref{etype}.

Abstract elementary classes (AECs for short) were introduced by Shelah \cite{sh88} in the mid-seventies to capture the semantic structure of non-first-order theories. The definition of AEC will mention the following logical notions:  a language $\tau$, $\tau$-structures and the substructure relation. In this paper, the language $\tau$ will always be $\{0, +,-\} \cup \{ r\cdot   : r \in R \}$ where $R$ is a fixed ring and $r \cdot$ is interpreted as multiplication by $r$ for every $r \in R$. So, a class of $\tau$-structures will be a class of $R$-modules for a fixed ring $R$. Moreover, being a substructure will mean being a submodule. The definitions of all of these notions can be found in \cite[\S I]{mar}. We will give many algebraic examples of AECs right after its definition. For a structure $M$, we denote by $|M|$ the underlying set of $M$ and by $\|M\|$ its cardinality.

\begin{defin}\label{aec-def}
  An \emph{abstract elementary class} is a pair $\K = (K, \lea)$, where:

  \begin{enumerate}
    \item $K$ is a class of $\tau$-structures, for some fixed language $\tau = \tau (\K)$. 
    \item $\lea$ is a partial ordering on $K$ extending the substructure relation.
    \item $(K, \lea)$ respects isomorphisms: If $M \lea N$ are in $K$ and $f: N \cong N'$, then $f[M] \lea N'$. In particular (taking $M = N$), $K$ is closed under isomorphisms.
    \item Coherence: If $M_0, M_1, M_2 \in K$ satisfy $M_0 \lea M_2$, $M_1 \lea M_2$, and $M_0 \subseteq M_1$, then $M_0 \lea M_1$.
    \item Tarski-Vaught axioms: Suppose $\delta$ is a limit ordinal and $\{ M_i \in K : i < \delta \}$ is an increasing chain. Then:

        \begin{enumerate}

            \item $M_\delta := \bigcup_{i < \delta} M_i \in K$ and $M_i \lea M_\delta$ for every $i < \delta$.
            \item\label{smoothness-axiom}Smoothness: If there is some $N \in K$ so that for all $i < \delta$ we have $M_i \lea N$, then we also have $M_\delta \lea N$.

        \end{enumerate}

    \item L\"{o}wenheim-Skolem-Tarski axiom: There exists a cardinal $\lambda \ge |\tau(\K)| + \aleph_0$ such that for any $M \in K$ and $A \subseteq |M|$, there is some $M_0 \lea M$ such that $A \subseteq |M_0|$ and $\|M_0\| \le |A| + \lambda$. We write $\LS (\K)$ for the minimal such cardinal.
  \end{enumerate}
\end{defin}

 Below we introduce many examples of AECs in the context of algebra. Recall that (for abelian groups) $G$ is a \emph{pure subgroup} of $H$, denoted   by $G \leq_p H$, if and only if $nH \cap G = nG$ for every $n \in \mathbb{N}$. For $R$-modules $M$ and $N$, we say that $M$ is a pure submodule of $N$ if for every $L$ right $R$-module $L \otimes M \to L \otimes N$ is a monomorphism.
 
\begin{example}\label{ex1} We begin by giving some examples of abstract elementary classes contained in the class of abelian groups:
\begin{itemize}
\item $\K_\leq^{Ab}:=(Ab, \leq)$ where $Ab$ is the class of abelian groups.
\item $\K^{Ab}:=(Ab, \leq_p)$ where $Ab$ is the class of abelian groups.
\item $\Kp=(\Kkp, \leq_p)$ where $\Kkp$ is the class of abelian $p$-groups for $p$ a prime number. A group $G$ is a $p$-group if every element $g \neq 0$ has order $p^n$ for some $n \in \mathbb{N}$. 
\item $\K^{Tor}=(\text{Tor}, \leq_p)$ where $Tor$ is the class of abelian torsion groups. A group $G$ is a torsion group if every element $g \neq 0$ has finite order. 
\item $\K^{tf}=(K^{tf}, \leq_p)$ where $K^{tf}$ is the class of torsion-free abelian groups. A group $G$ is a torsion-free group if every element has infinite order.
\item $\K^{rtf}=(K^{rtf}, \leq_p)$ where $K^{rtf}$ is the class of
reduced torsion-free abelian groups. A group $G$ is reduced if it does not have non-trivial divisible subgroups.
\item $\K^{B_0}=(K^{B_0}, \leq_p)$ where $K^{B_0}$ is the class of finitely Butler groups. A group $G$ is a finitely Butler group if $G$ is torsion-free and every pure subgroup of finite rank is a pure subgroup of a finite rank completely decomposable group (see \cite[\S 14.4]{fuc} for more details).
\item $\K^{\aleph_1\text{-free}}=(\aleph_1\text{-free}, \leq_p)$ where $\aleph_1\text{-free}$ is the class of $\aleph_1\text{-free}$ groups. A group $G$ is $\aleph_1$-free if every countable subgroup of $G$ is free.

\end{itemize}
Many of these examples can be extended to arbitrary rings. Below are some examples of AECs in classes of modules:
\begin{itemize}
\item $(R\text{-Mod}, \subseteq_R)$ where $R$-Mod is the class of $R$-modules.
\item $(R\text{-Mod}, \leq_{p})$ where $R$-Mod is the class of $R$-modules.
\item $(R\text{-Flat}, \leq_{p})$ where $R$-Flat is the class of flat $R$-modules. An $R$-module $F$ is flat if $(-) \otimes F$ is an exact functor. 
%\item $(R\text{-Absp}, \subseteq_R)$ where $R\text{-Absp}$ is the class of absolutely pure modules\footnote{$N$ is an absolutely pure module, if every $N$ if $M \subseteq_R N$, then $M \leq_p N$.}.
\end{itemize}
\end{example}

Let us now introduce some notation.

\begin{nota}\
\begin{itemize}
\item For $\lambda$ an infinite cardinal, $\K_{\lambda} =\{ M \in K : \| M \| =\lambda \}$.
\item Let $M, N \in K$. If we write ``$f: M \to N$" we assume that $f$ is a $\K$-embedding, i.e., $f: M \cong f[M]$ and $f[M] \lea N$. In the classes studied in this paper a $\K$-embedding is either a monomorphism or a pure monomorphism.
\end{itemize}
\end{nota}

The next three properties are properties that an AEC may or may not have. The first one is a weakening of the notion of pushout. 

\begin{defin}\
\begin{itemize}
\item $\K$ has the amalgamation property if for every $M, N_1, N_2 \in K$ such that $M \lea N_1, N_2$ there are $N \in K$, and $\K$-embeddings $f_1: N_1 \to N$ and $f_2: N_2 \to N$ such that $f_1\rest_{M}= f_2\rest_M$.
\item $\K$ has the joint embedding property if for every $M, N \in K$ there are $L \in K$ and $\K$-embeddings $f, g$ such that $f: M \to L$ and $g: N \to L$. 
\item $\K$ has no maximal models if for every $M \in K$, there is $N \in K$ such that $M \lea N$ and $M \neq N$. 
\end{itemize}
\end{defin} 

\begin{example}\label{ex2}
All the AECs introduced in Example \ref{ex1} have joint embedding and no maximal models, this is the case as they are all closed under direct sums. As for the amalgamation property, all the AECs introduced in Example \ref{ex2} have it with the exception $\K^{\aleph_1\text{-free}}$, and the possible exception of $\K^{B_0}$. The problem remains open for the latter case. 
\end{example}

The next notion was introduced by Shelah in \cite{sh300}, it extends the notion of  first-order type to this setting. It is one of the key notions that let us answer the Main Problem.

\begin{defin}\label{gtp-def}
  Let $\K$ be an AEC.
  
  \begin{enumerate}
    \item Let $\K^3$ be the set of triples of the form $(\bb, A,
N)$, where $N \in K$, $A \subseteq |N|$, and $\bb$ is a sequence
of elements from $N$. 
    \item For $(\bb_1, A_1, N_1), (\bb_2, A_2, N_2) \in \K^3$, we
say $(\bb_1, A_1, N_1)E_{\text{at}}^{\K} (\bb_2, A_2, N_2)$ if $A
:= A_1 = A_2$, and there exist $\K$-embeddings $f_\ell : N_\ell \to_A N$ for $\ell \in \{ 1, 2\}$ such that
$f_1 (\bb_1) = f_2 (\bb_2)$ and $N \in \K$.
    \item Note that $E_{\text{at}}^{\K}$ is a symmetric and
reflexive relation on $\K^3$. We let $E^{\K}$ be the transitive
closure of $E_{\text{at}}^{\K}$.
    \item For $(\bb, A, N) \in \K^3$, let $\gtp_{\K} (\bb / A;
N) := [(\bb, A, N)]_{E^{\K}}$. We call such an equivalence class a
\emph{Galois-type}. Usually, $\K$ will be clear from the context and we will omit it.
\item For $M \in K$, $\gS_{\K}(M)= \{  \gtp_{\K}(b / M; N) : M
\leq_{\K} N\in K \text{ and } b \in N\} $. 
  \end{enumerate}
\end{defin}

\begin{example}\label{etype}
This is the only place where we assume the reader is familiar with basic logic notions. The necessary logic background is presented in \cite[\S 2.2]{kuma}. 

\begin{itemize}
\item ( \cite[3.12]{maz20} ) In $\K = \K_\leq^{Ab}$ we have that for any $G, G_1, G_2$ in $\K$ with  $G \leq G_1, G_2$, $\bar{b}_{1}
\in  G_1^{<\omega}$  and $\bar{b}_{2} \in G_2^{<\omega}$,
 \[ \gtp_\K(\bar{b}_{1}/G; G_1) = \gtp_\K(\bar{b}_{2}/G; G_2) \text{ if
and
only if } qf\text{-}tp(\bar{b}_{1}/G , G_1) = qf\text{-}tp(\bar{b}_{2}/G, G_2). \]
Where $qf\text{-}tp(\bar{b}_{1}/G , G_1)$ is the set of quantifier-free formulas with parameters in G that hold for $\bar{b}_{1}$ in $G_1$.
\item ( \cite[3.14]{kuma} ) In $\K= \K^{Ab}$ or $\K= \K^{tf}$ we have that for $G, G_1, G_2$ in $\K$ with  $G \leq_{p} G_1, G_2$, $\bar{b}_{1}
\in  G_1^{<\omega}$  and $\bar{b}_{2} \in G_2^{<\omega}$,
 \[ \gtp_\K(\bar{b}_{1}/G; G_1) = \gtp_\K(\bar{b}_{2}/G; G_2) \text{ if
and
only if } \pp(\bar{b}_{1}/G , G_1) = \pp(\bar{b}_{2}/G, G_2).\]

Where $\pp(\bar{b}_{1}/G , G_1)$ is the set of positive primitive formulas with parameters in G that hold for $\bar{b}_{1}$ in $G_1$.

\end{itemize}

A direct consequence of Lemma \ref{types} (see below) is that the second bullet is also true for $\Kp$ for any prime number $p$.

On a different direction, it is not known if the analogue is true for the class $(R\text{-Flat}, \leq_p)$. It is shown that this is the case under extra hypothesis in \cite[4.4]{maz2}.
\end{example}

One of the main objectives of the theory of AECs is to find dividing lines analogous to those of first-order theories. A dividing line is a property such that the classes satisfying such a property have some nice behaviour while those not satisfying it have a bad one. An introduction to dividing lines for mathematicians not working in mathematical logic can be found in \cite[Part I]{sh1151} and \cite{balc}. One of the first dividing lines that was studied is that of stability.

\begin{defin}
 An AEC $\K$ is \emph{$\lambda$-stable}  if for any $M \in
\K_\lambda$, $| \gS_{\K}(M) | \leq \lambda$. An AEC is stable if there is a $\lambda \geq \LS(\K)$ such that $\K$ is $\lambda$-stable.
\end{defin}

As the following example will be used in the main section of the paper we recall it as a fact. 

\begin{fact}\label{exstab}\
\begin{enumerate}
\item (\cite{grp}, \cite[3.12]{maz1} ) $\K^{Ab}_\leq$ is $\lambda$-stable for every $\lambda \geq \aleph_0$.
\item (\cite[3.16]{kuma} ) $\K^{Ab}$ is $\lambda$-stable for every $\lambda$ such that $\lambda^{\aleph_0} = \lambda$.
\item  (\cite[0.3]{baldwine}, \cite[1.2]{sh820}, \cite[5.9]{maz20} ) $\K^{tf}$, $\K^{rtf}$ and $\K^{B_0}$ are $\lambda$-stable for every $\lambda$ such that $\lambda^{\aleph_0} = \lambda$.
\item $\K^{\aleph_1\text{-free}}$  is $\lambda$-stable for every $\lambda$ such that $\lambda^{\aleph_0} = \lambda$.
\item (\cite[3.6]{maz1}, \cite[3.16]{kuma} ) $(R\text{-Mod}, \subseteq_R)$ and $(R\text{-Mod}, \leq_{p})$ are $\lambda$-stable for every $\lambda$ such that $\lambda^{|R| +\aleph_0} = \lambda$.
\item (\cite[6.20, 6.21]{lrv} ) $(R\text{-Flat}, \leq_{p})$ is stable.
\end{enumerate}
\end{fact} 
\begin{proof}[Proof sketch] The only bullet point that is missing references is (4). The proof is similar to the one of $\K^{B_0}$ (see \cite[5.8]{maz20}). It follows from the fact that $\aleph_1$-free groups are closed under pure subgroups.  \end{proof}

\begin{remark}
We will obtain in Lemma \ref{psta} (see below) that $\Kp$ is also $\lambda$-stable if $\lambda^{\aleph_0} = \lambda$. 
\end{remark}

As witnessed by Fact \ref{exstab}, Lemma \ref{psta} and Corollary \ref{last}, all the AECs of abelian groups and modules identified in this paper are stable. Moreover, a classical result of first-order model theory assures us that:

\begin{fact}[Fisher, Baur, see e.g. {\cite[3.1]{prest}}]\label{st-c}
If $T$ is a complete first-order theory extending the theory of modules, then $(Mod(T), \leq_p)$ is stable.  
\end{fact}

The same result is obtained for incomplete first-order theories of modules closed under direct sums in \cite[3.16]{kuma}. So we ask if the above is still true for all AECs of modules.

\begin{question} 
Let $R$ be an associative ring with an identity element.

If $(K, \leq_p)$ is an AEC such that $K \subseteq R\text{-Mod}$, is $(K, \leq_p)$ stable? Is this true if $R= \mathbb{Z}$?
\end{question}

Let us recall the classical notion of a universal model. 
\begin{defin}
Let $\K$ be an AEC and $\lambda$ be an infinite cardinal. We say that $M \in \K$ is a \emph{universal model in
$\K_\lambda$} if $M \in \K_\lambda$ and if given any $N \in \K_\lambda$, there is a $\K$-embedding $f: N \to M$. We say that $\K$ has a universal model of cardinality $\lambda$ if there is a universal model in $\K_\lambda$. 
\end{defin}

\begin{remark} Let $p$ be a prime number. Observe that the phrase \emph{there is a universal abelian $p$-group for purity of cardinality $\lambda$} means precisely that \emph{$\Kp$ has a universal model of cardinality $\lambda$}. We will use the latter in the rest of the paper.
\end{remark}

One of the nice consequences of being stable is that it is easy to build universal models in many cardinals.

\begin{fact}[{\cite[3.20]{kuma}}]\label{universal}
Let $\K$ be an AEC with joint embedding, amalgamation and no maximal
models. Assume there is $\theta_0 \geq \LS(\K)$ and $\kappa$ such
that for all $\theta \geq \theta_0$, if $\theta^{\kappa} =
\theta$, then $\K$  is $\theta$-stable.

Suppose $\lambda > \theta_0$. If $\lambda^{\kappa}=\lambda$ or
$\forall \mu < \lambda( \mu^{\kappa} < \lambda)$, then $\K$
has a universal model of cardinality $\lambda$.
\end{fact} 

This is all the theory of AECs that is required to follow the proof of the main theorem (Theorem \ref{main}). Connections between algebra and AECs were studied in: \cite{grsh0}, \cite{grsh}, \cite{baldwine}, \cite{satr}, \cite{maz20}, \cite{kuma}, \cite{grp}, \cite[\S 5]{bontor}, \cite[\S 6]{lrv}, \cite{maz1} and \cite{maz2}. Finally, the reader interested in a more in depth introduction to AECs can consult: \cite{grossberg2002}, \cite{baldwinbook09} and \cite{shelahaecbook}.

\section{Main result}

In this section we will study the class of abelian $p$-groups with pure embeddings for any prime number $p$. We introduced these classes as the third bullet point of Example \ref{ex1}. Following the notation of Example \ref{ex1}, we will denote them by $\Kp$ for any prime number $p$.

The following assertion contains many known facts about abelian $p$-groups. 
\begin{fact}\label{bas} Assume $p$ is a prime number.
\begin{enumerate}
\item $\Kp=(\Kkp, \leq_p)$ is an AEC with $\LS(\Kp)= \aleph_0$.
\item $\Kp$ has the amalgamation property, the joint embedding property and no maximal models.
\item The class of abelian $p$-groups is not closed under pure-injective envelopes.
\item The class of abelian $p$-groups is not first-order axiomatizable. 

\end{enumerate}
\end{fact}
\begin{proof}
(1) is trivial and (2) follows from the closure of abelian $p$-groups under direct sums and pushouts. As for closure under pure-injective envelopes, recall that the pure-injective envelope of  $\oplus_n \mathbb{Z}(p^n)$ is $\Pi_n \mathbb{Z}(p^n)$, which is not an abelian $p$-group. That the class of abelian $p$-groups is not first-order axiomatizable follows from the last line together with the fact that the pure-injective envelope of a module is an elementary extension.
\end{proof}

\begin{remark}
Items (3) and (4) of the above assertion will not be used, but hint to the difficulty of dealing with this class from a model theoretic perspective.
\end{remark}

\begin{remark}
It is worth mentioning that if the Generalized Continuum Hypothesis (GCH) holds\footnote{GCH states that $2^\lambda = \lambda^+$ for every infinite cardinal $\lambda$.}, then just from the fact that $\Kp$ is an AEC with amalgamation, joint embedding and no maximal models, one can easily show that $\Kp$ has a universal model of cardinality $\lambda$ for every $\lambda$ uncountable cardinal. 
\end{remark}

We begin by characterizing Galois-types in the class of abelian $p$-groups.  Recall, from Example \ref{ex1}, that $\K^{Ab}$ is the class of abelian groups with pure embeddings. 

\begin{theorem}\label{ptypes} Assume $p$ is a prime number.
 Let $G_1, G_2 \in \Kkp$, $A \subseteq G_1, G_2$, $\bar{b}_1 \in G_1^{< \omega}$ and $ \bar{b}_2 \in G_2 ^{< \omega}$, then:
\[ \gtp_{\Kab}(\bar{b}_1/A; G_1) = \gtp_{\Kab}(\bar{b}_2/A; G_2) \text{ if and only if } \gtp_{\Kp}(\bar{b}_1/A; G_1) = \gtp_{\Kp}(\bar{b}_2/A; G_2).\]
\end{theorem}
\begin{proof}
The backward direction is trivial as $\Kab$ and $\Kp$ are both AECs with respect to pure embeddings. We prove the forward direction. Assume that  $\gtp_{\Kab}(\bar{b}_1/A; G_1) = \gtp_{\Kab}(\bar{b}_2/A; G_2)$, then by definition of Galois-types there are $G \in K^{Ab}$, $f_1: G_1 \to G$ and $f_2: G_2 \to G$ such that $f_1\rest_A=f_2\rest_A$ and $f_1(\bb_1)=f_2(\bb_2)$. Let $L = f[G_1] + f[G_2] \leq G$. Observe that $f[G_1] + f[G_2]$ is an abelian $p$-group and that$f[G_1], f[G_2] \leq_p f[G_1] + f[G_2]$ since $f_1, f_2$ are pure embeddings. Therefore, the following square is a commutative square in $\Kp$: 
\[
 \xymatrix{\ar @{} [dr] G_2  \ar[r]^{f_1 \hspace{1cm}}  & f_1[G_1] + f_2[G_2] \\
A \ar[u]^{\id} \ar[r]_{\id} & G_2 \ar[u]_{f_2}
} 
\]

and $f_1(\bb_1)=f_2(\bb_2)$. Hence $\gtp_{\Kp}(\bar{b}_1/A; G_1) = \gtp_{\Kp}(\bar{b}_2/A; G_2)$. \end{proof}

We show that $\Kp$ is stable for every prime number $p$.

\begin{lemma}\label{psta} Assume $p$ is a prime number.
If $\lambda^{\aleph_0}= \lambda$, then $\Kp$ is $\lambda$-stable.
\end{lemma}
\begin{proof}
Let $G \in \Kkp_\lambda$ and  $\{ \gtp_{\Kp}(b_i/G ; G_i) : i < \alpha \}$ be an enumeration without repetitions of $\gS_{\Kp}(G)$. Let $\Phi: gS_{\Kp}(G) \to gS_{\Kab}(G)$ be given by $\Phi(\gtp_{\Kp}(b_i/G ; G_i))=\gtp_{\Kab}(b_i/G ; G_i)$. Then by Theorem \ref{ptypes}, $\Phi$ is a well-defined injective function, so $| \gS_{\Kp}(G) | \leq |gS_{\Kab}(G)|$. Then as $\Kab$ is $\lambda$-stable, by Fact \ref{exstab}.(2), we conclude that $| \gS_{\Kp}(G) | \leq \lambda$.
\end{proof}

With this we are able to obtain many cardinals such that there are universal models for abelian $p$-groups for purity.

\begin{theorem}\label{main} Let $p$ be a prime number.
If $\lambda^{\aleph_0}=\lambda$ or $\forall \mu < \lambda(
\mu^{\aleph_0} < \lambda)$, then $\Kp$ has a universal model of cardinality $\lambda$.
\end{theorem}
\begin{proof}
By Lemma \ref{psta} $\Kp$ is $\lambda$-stable if $\lambda^{\aleph_0}= \lambda$. As $\Kp$ has amalgamation, joint embedding and no maximal models by Fact \ref{bas}, it follows by Fact \ref{universal} that there is a universal model of cardinality $\lambda$ if  $\lambda^{\aleph_0}=\lambda$ or $\forall \mu < \lambda(
\mu^{\aleph_0} < \lambda)$. \end{proof}

\begin{remark}
Recall that for an abelian group $G$, $t_p(G)= \{ g \in G : \text{ there exists an } n \in \mathbb{N} \text{ s.t. } p^n g = 0 \}$. One can show that if $G \in K^{Ab}$ is a universal group in $\Kab_\lambda$, then $t_p(G) \in K^{p\text{-grp}}$ is a universal group in $\Kp_\lambda$. In light of this observation, the above result also follows from \cite[3.19]{kuma}. The reason we decided to present the above  argument for Theorem \ref{main} is to showcase how the technology of AECs can be used to obtain universal models. Moreover, the method presented in this section gives us more information about $\Kp$ (it shows that $\Kp$ is stable) and can be generalized to other classes (see Section 4).
\end{remark}

The above theorem is the main result in the positive direction. The next result shows that if certain inequalities hold, then there are cardinals where there are no universal models.

\begin{lemma}\label{nouniv}
Let $\lambda$ be a regular cardinal and $\mu$ be an infinite cardinal. If $\mu^+ < \lambda < \mu^{\aleph_0}$, then $\Kp$ does not have a universal model of cardinality $\lambda$. 
\end{lemma}
\begin{proof}[Proof sketch]
In \cite[3.3]{kojsh} it is shown that there is no universal model of cardinality $\lambda$ in the class of separable abelian $p$-groups for purity. What is presented there can be used to conclude the stronger result that there is no universal model of cardinality $\lambda$ in the class of abelian $p$-groups for purity.
\end{proof}

As a simple corollary we obtain: 

\begin{cor}\label{bome}
For $n\geq 2$, $\Kp$ has a universal model of cardinality $\aleph_n$ if and only if $2^{\aleph_0} \leq \aleph_n$. 
\end{cor}
\begin{proof}
\fbox{$\to$} Assume for the sake of contradiction that $2^{\aleph_0} > \aleph_n$, then $\aleph_0^+ < \aleph_n < \aleph_0^{\aleph_0}$. So we get a contradiction by Lemma \ref{nouniv}.

\fbox{$\leftarrow$} If $2^{\aleph_0} \leq \aleph_n$, then $\aleph_n^{\aleph_0}=2^{\aleph_0}\aleph_n= \aleph_n$ where the first equality follows from Hausdorff formula. Therefore, by Theorem \ref{main}, there is a universal model of cardinality $\aleph_n$. 
\end{proof}

The above corollary gives a complete solution to the Main Problem below $\aleph_\omega$ except for the cases of $\aleph_0$ and $\aleph_1$. The next lemma addresses the case of $\aleph_1$. 

\begin{lemma}\label{one}\
\begin{enumerate}
\item If $2^{\aleph_0} = \aleph_1$, then $\Kp$ has a universal model of cardinality $\aleph_1$.
\item If $2^{\aleph_0} > \aleph_1$ and the combinatorial principle $\clubsuit$ holds\footnote{$\clubsuit$ is a combinatorial principle similar to $\lozenge$, but weaker. For the definition and what is known about $\clubsuit$, the reader can consult \cite[\S I.7]{fsh}. }, then $\Kp$ does not have a universal model of cardinality $\aleph_1$.
\end{enumerate}
\end{lemma}
\begin{proof}
(1) follows directly from Theorem \ref{main}, so we focus on (2). It is easy to show that $\clubsuit$ implies the existence of a club guessing sequence in $\aleph_1$ in the sense of \cite[1.5]{kojsh}. Then the result follows from \cite[3.3]{kojsh} by a similar argument to the one given in Lemma \ref{nouniv}. 
\end{proof}

So for cardinals below $\aleph_\omega$ we are only left with $\aleph_0$.  So we ask:

\begin{question}\label{quest}
Is there a universal model in $\Kp$ of cardinality $\aleph_0$?
\end{question}

\begin{remark}
In the case of $\aleph_0$, we think that the question can be answered using group theoretic methods. For instance, using Exercise 6 of \cite[\S 11.1]{fuc} it is possible to build an abelian $p$-group $G$ of cardinality $\aleph_1$ such that every countable abelian $p$-group is purely embeddable into $G$. 
\end{remark}

\begin{remark}
For cardinals greater than or equal to $\aleph_\omega$, Theorem \ref{main} gives many instances for existence of universal models in $\Kp$. For example, Theorem \ref{main} implies the existence of a universal model in $2^{\aleph_1}$ or $\beth_\omega$. On the other hand, Lemma \ref{nouniv} gives instances where there are no universal models if GCH fails. There are still some cardinals that are not covered by any of these cases, but based on what is known about reduced torsion-free groups and separable torsion groups (see \cite[\S 10.(B)]{sh1151}, in particular Table 1 at the end of that paper), we think that the answer in those cases depends even more on set theoretic hypotheses and not on the theory of AECs. 
\end{remark}

\section{Some generalizations}

The key ideas we used to understand $\Kp$ were that $\Kp$ fits nicely inside $\Kab$ and that the class $\Kab$ is well-understood and well-behaved. As we think that these ideas might have further applications, we abstract this set up in the next definition.

\begin{defin}
Let $\K= (K, \lea)$ and $\K^{\star}=(K^{\star}, \lea)$ be a pair of AECs. We say that $\K^{\star}$ is nicely generated inside $\K$ if:
\begin{enumerate}
\item $K^{\star} \subseteq K$.
\item For any $N_1, N_2 \in K^{\star}$ and $N \in K$, if $N_1, N_2 \lea N$, then there is $L \in K^{\star}$ such that $N_1, N_2 \lea L \subseteq N$.
\end{enumerate}
\end{defin}

\begin{example}\
\begin{enumerate}
\item  $\Kp$ is nicely generated inside $\Kab$.  Given $N_1, N_2 \in \Kkp$ and $N \in Ab$ such that $N_1, N_2 \leq_p N$, then $L= N_1+N_2 \in \Kkp$ and it satisfies that $N_1, N_2 \leq_p L \subseteq N$. 
\item $\K^\text{Tor}$ is nicely generated inside $\Kab$.
%\item $\K^{tf}$ is nicely generated inside $\Kab$.
\item  $\Kp_\leq=(\Kkp, \leq)$ and $\K^\text{Tor}_\leq=(K^\text{Tor}, \leq)$ are nicely generated inside $\K^{Ab}_\leq$.
\item $\K^{R\text{-Tor}}=(K^{R\text{-Tor}}, \leq_p)$ is nicely generated inside $(R\text{-Mod}, \leq_{p})$ where $R$ is an integral domain and $K^{R\text{-Tor}}$ is the class of $R$-torsion modules. An $R$-module $M$ is an $R$-torsion module if for every $m \in M$ there is $r \neq 0$ such that $rm= 0$.
\item $\K^{R\text{-Div}} = (K^{R\text{-Div}}, \leq_p)$ is nicely generated inside $(R\text{-Mod}, \leq_{p})$ where $R$ is an integral domain and $K^{R\text{-Div}}$ is the class of $R$-divisible modules. An $R$-module $M$ is an $R$-divisible module if for every $m \in M$ and $r \neq 0$, there is an $n \in M$ such that $rn= m$.
\item $\K^\text{Div}_\leq = (K^\text{Div}, \leq)$ is nicely generated inside $\K^{Ab}_\leq$, where $K^\text{Div}$ is the class of divisible abelian groups.

\end{enumerate}
\end{example}

The next lemma is the key observation.

\begin{lemma}\label{com-square}
Assume $\K^{\star}=(K^{\star}, \lea)$ is nicely generated inside $\K= (K, \lea)$. If $M\lea N_1, N_2 \in K^{\star}$ and there are $N' \in K$ and $\K$-embeddings $f_1: N_1 \to N'$ and $f_2: N_2 \to N'$ such that $f_1\rest_M=f_2\rest_M$, then there are $L \in K^\star$ and $\K^\star$-embeddings $g_1: N_1 \to L$ and $g_2: N_2 \to L$ such that $L \subseteq N'$ and $i \circ g_\ell = f_\ell$ for $\ell \in \{ 1, 2 \}$ where $i: L \to N'$ is the inclusion. 
%\[
  %\xymatrix@=3pc{
    %& & N' \\
    %N_1 \ar@{-->}^{g_1}[r] \ar@/^/[rru]^{f_1} & L \ar@{-->}[ru]^i & \\
    %M \ar [u] \ar[r] & N_2 \ar@{-->}[u]_{g_2} \ar@/_/[ruu]_{f_2} &
  %}
%\]
\end{lemma}
%\begin{proof}
%As $\K^{\star}$ is nicely generated inside $\K$ there is $L \in K^\star$ such that $f_1[N_1], f_2[N_2] \lea L \subseteq N'$. So it follows that the smaller square given by $f_1: N_1 \to L$ and $f_2: N_2 \to L$ is a commutative square in $\K^\star$.  \end{proof}

\begin{cor} Assume $\K^{\star}=(K^{\star}, \lea)$ is nicely generated inside $\K= (K, \lea)$ .  If $\K$ has the amalgamation property, then $\K^{\star}$ has the amalgamation property.
\end{cor}

The next two result are proven by generalizing the proofs of Theorem \ref{ptypes} and Theorem \ref{psta} respectively. 

\begin{cor}\label{types}
Assume $\K^{\star}=(K^{\star}, \lea)$ is nicely generated inside $\K= (K, \lea)$. Let $N_1, N_2 \in K^\star$, $A \subseteq N_1, N_2$, $\bar{b}_1 \in N_1^{< \omega}$ and $ \bar{b}_2 \in N_2^{< \omega}$, then:
\[ \gtp_{\K}(\bar{b}_1/A; N_1) = \gtp_{\K}(\bar{b}_2/A; N_2) \text{ if and only if } \gtp_{\K^\star}(\bar{b}_1/A; N_1) = \gtp_{\K^\star}(\bar{b}_2/A; N_2)\]
\end{cor}
%\begin{proof}
%The backward direction is clear as a commuting square in $\K^\star$ is a commuting square in $\K$ so we prove the forward direction. We may assume for simplicity that $\K$ has the amalgamation property. By definition of Galois-types there are $N' \in K$, $f_1: N_1 \to N'$ and $f_2: N_2 \to N'$ such that $f_1\rest_M=f_2\rest_M$ and $f_1(\bb_1)=f_2(\bb_2)$. Then by Lemma \ref{com-square} there is $L \in K^\star$ such that $f_1: N_1 \to L$ and $f_2: N \to L$ becomes a commutative square in $\K^\star$. But observe that moreover, $f_1(\bb_1)=f_2(\bb_2)$ so $\gtp_{\K^\star}(\bar{b}_1/A; N_1) = \gtp_{\K^\star}(\bar{b}_2/A; N_2)$.
%\end{proof}

\begin{theorem}
Assume $\K^{\star}=(K^{\star}, \lea)$ is nicely generated inside $\K= (K, \lea)$. If $\K$ is $\lambda$-stable in $\lambda \geq \LS(\K^\star)$, then $\K^\star$ is $\lambda$-stable. 
\end{theorem}

\begin{remark} Another consequence of Corollary \ref{types} is that if $\K^{\star}=(K^{\star}, \lea)$ is nicely generated inside $\K= (K, \lea)$ then: if $\K$ is $(< \aleph_0)$-tame then $\K^\star$ is $(< \aleph_0)$-tame. Therefore, we get that $\Kp$ is $(< \aleph_0)$-tame as $\Kab$ is $(< \aleph_0)$-tame by \cite[3.15]{kuma}. Recall that an AEC is $(< \aleph_0)$-tame if its Galois-types are determined by their restrictions to finite sets (the reader can consult the definition in \cite[1.6]{maz20}).
\end{remark}

Finally, these results can be used to obtain stability cardinals and universal models.

\begin{cor}\label{last} Let $p$ be a prime number and $R$ be an integral domain.
\begin{enumerate}
\item If $\lambda^{\aleph_0}=\lambda$, then $\Kp$ and $\K^{Tor}$ are $\lambda$-stable. If $\lambda^{\aleph_0}=\lambda$ or $\forall \mu < \lambda(\mu^{\aleph_0} < \lambda)$, then $\Kp$ and $\K^{Tor}$ have a universal model of cardinality $\lambda$.
\item If $\lambda^{|R| + \aleph_0}=\lambda$, then $\K^{R\text{-Tor}}$ and $\K^{R\text{-Div}}$ are $\lambda$-stable. If $\lambda^{|R| + \aleph_0}=\lambda$ or $\forall \mu < \lambda(\mu^{|R| + \aleph_0} < \lambda)$, then $\K^{R\text{-Tor}}$ and $\K^{R\text{-Div}}$ have a universal model of cardinality $\lambda$.
\item $\Kp_\leq$, $\K^\text{Tor}_\leq$ and $\K^\text{Div}_\leq$ are $\lambda$-stable and have a universal model of cardinality $\lambda$ for every $\lambda \geq \aleph_0$. 
\end{enumerate}
\end{cor}

\begin{remark}
%The same solution for torsion-free groups with pure embeddings is obtained in \cite[3.22]{kuma} using similar methods and in \cite[1.2]{sh820} using different methods.
Universal models for torsion groups with embeddings and divisible groups with embeddings can be obtained by algebraic methods. More precisely, $\oplus_{p \text{ prime}} (\oplus_{\lambda} \mathbb{Z}(p^\infty))$ is the universal model of size $\lambda$ for torsion groups and $\oplus_{p \text{ prime}} (\oplus_{\lambda} \mathbb{Z}(p^\infty)) \oplus \oplus_\lambda \mathbb{Q}$ is the universal model of size $\lambda$ for divisible groups.
\end{remark}

%\bibliography{}{}
%\bibliographystyle{amsalpha}

\end{document}